\documentclass[12pt]{amsart}
\usepackage{amsthm}
\usepackage{amssymb}
\usepackage{amsmath}
\usepackage[shortlabels]{enumitem}

\newtheorem{theorem}{Theorem}[section]

\theoremstyle{definition}

\newtheorem{example}{Example}[section]
\newtheorem{assumption}{Assumption}[section]
\theoremstyle{remark}

\newcommand{\brref}[1]{(\ref{#1})}

\newcommand{\milano}{Dipartimento di Matematica ``F. Enriques"
 \\ Universit\`a degli Studi di Milano \\ Via Saldini 50 \\ 20133
Milano, Italy}

\newcommand{\Pin}[1]{{\mathbb P}^{#1}}

\newcommand{\rk}[1]{{\rm rk}\,(#1)}

\newcommand{\nbXi}{\mathbf{X}_i}

\newcommand{\Phicols}[2]{[\Phi^k_{h_1,h_2,h_3}]_#1^#2}
\newcommand{\T}{\mathcal{T}}
\newcommand{\tensor}{\otimes}
\makeatletter
\renewcommand*\env@matrix[1][*\c@MaxMatrixCols c]{%
  \hskip -\arraycolsep
  \let\@ifnextchar\new@ifnextchar
  \array{#1}}
\makeatother
\title{The Rank of Trifocal Grassmann Tensors}

\author[M.Bertolini]{Marina Bertolini}
\email{marina.bertolini@unimi.it}
\address{\milano}

\author[G.M.Besana]{Gian Mario Besana}
\email{gbesana@depaul.edu}
\address{College of Computing and Digital Media \\ DePaul University \\ 243 South Wabash \\ Chicago IL, 60604 USA}

\author[G.Bini]{Gilberto Bini}
\email{gilberto.bini@unimi.it}
\address{\milano}

\author[C.Turrini]{ Cristina Turrini}
\email{cristina.turrini@unimi.it}
\address{\milano}

\makeatletter
\renewcommand*\env@matrix[1][*\c@MaxMatrixCols c]{%
  \hskip -\arraycolsep
  \let\@ifnextchar\new@ifnextchar
  \array{#1}}
\makeatother

\begin{document}
\date{\today}

%%%%%%%%%%%%%%%%%%%%%%%%%%%%%%%%%%%%%%%%%%%%%%%%%%%%%%%
\begin{abstract}
Grassmann tensors arise from classical problems of scene
reconstruction in computer vision. Trifocal Grassmann tensors,
related to three projections from a projective space of dimension
$k$ onto view-spaces of varying dimensions are studied in this
work. A canonical form for the combined projection matrices is
obtained. When the centers of projections satisfy a natural
generality assumption, such canonical form gives a closed formula
for the rank of trifocal Grassmann tensors. The same
approach is also applied to the case of two projections,
confirming a previous result obtained with different methods in
\cite{tubbAMPA}. The rank of sequences of tensors converging to
tensors associated with degenerate configurations of projection
centers is also considered, giving concrete examples of a wide
spectrum of phenomena that can happen.

\end{abstract}

\maketitle

\section{Introduction}

Tensors, as multidimensional arrays representing multilinear
applications among vector spaces, have traditionally played a
pivotal role in many areas, from physics to computer science, to
electrical engineering. As algebraic geometry is increasingly
witnessing intense activity in more applied directions, tensors
have come to the fore of the discipline as useful tools on one
hand, and as beautifully intricate objects of study on the other,
with rich geometric interplay with other classical ideas. In
particular, the calculation of any of the various established
notions of rank of a tensor is an interesting and difficult
problem. While many authors have recently studied these issues, a
standard reference is \cite{LA} and a useful survey is
\cite{be-car-cat-gi-on}.

The authors have been interested for a while in a class of tensors
that arise from classical problems of scene reconstruction in
computer vision. In the classical case of reconstruction of a
three-dimensional static scene from two, three, or four
two-dimensional images, these tensors are known as the fundamental
matrix, the trifocal tensor, and the quadrifocal tensor,
respectively, and have been studied extensively, see for example
\cite{Hart-Zi2}, \cite{oe1}, \cite{oe2}, \cite{al-to2},
\cite{hey1}. In a more general setting, these tensors are called
{\it Grassmann} tensors and were introduced by Hartley and
Schaffalitzky, \cite{Hart-Schaf}, as a way to encode information
on corresponding subspaces in multiview geometry in higher
dimensions. Three of the authors have studied critical loci
for projective reconstruction from multiple views, \cite{tubbLAIA},
\cite{be-tur1}, and in this setting Grassmann tensors play a fundamental role, \cite{ber-tu-no1}, \cite{be-ber-no-tu}.

The authors' long-term goal is to study properties such as rank,
decomposition, degenerations, and identifiability of Grassmann
tensors in higher dimensions, and, when feasible, the varieties
parameterizing such tensors.

The first step was taken in \cite{tubbAMPA}, where three of the
authors studied the case of two views in higher dimensions,
introducing the concept of generalized fundamental matrices as
2-tensors. That first work contained an explicit geometric
interpretation of the rational map associated to the generalized
fundamental matrix, the computation of the rank of the generalized
fundamental matrix with an explicit, closed formula, and the
investigation of some properties  of the variety of such objects.

The next natural step in the authors' program is the study of
trifocal Grassmann tensors, i.e. Grassmann tensors arising from
three projections from higher dimensional projective spaces onto
view-spaces of varying dimensions. A natural genericity
assumption, see Assumption \ref{g.a.}, allows for suitable changes
of coordinates in the view spaces and in the ambient space that
give rise to a canonical form for the combined projection
matrices. Utilizing such canonical form, the rank of
trifocal Grassmann tensors is computed with a closed formula, see Theorem
\ref{therank}. When Assumption \ref{g.a.} is no longer satisfied,
the situation becomes quite intricate. A general canonical form
for the combined projection matrices can still be obtained, see
Section \ref{uprank}. We conclude with a series of examples in
which the rank is computed utilizing the canonical form. These
examples illustrate the wide spectrum of possible phenomena that
can happen with the specialization of the three centers of
projection. In particular, we provide examples of sequences of
Grassmann tensors of given rank $r$, converging to limit tensors
whose rank can be either strictly larger than $r,$ Example
\ref{examplep3p2p2p2deg}, and Example \ref{examplep4p2p2p2}-a, or
strictly smaller than $r,$ Example \ref{examplep4p2p2p2}-b. The
first two of these cases are geometric examples of tensors with
border rank strictly smaller than their rank.

\section{Background Material}

\subsection{Preliminaries on tensors} \label{introranks}
Notation and definitions of tensors
and their ranks (rank and border-rank) used in this work are relatively
standard in the literature. They are all contained in
\cite{LA} and briefly summarized below.

Given vector spaces $V_i, i = 1,\dots t,$  the {\it rank} of a tensor $T \in V_1 \tensor V_2\tensor ... \tensor V_t,$ denoted by $R(T),$ is the minimum number of decomposable tensors needed to write $T$ as a sum.
Recall that $R(T)$ is invariant under changes of bases in the vector spaces $V_i $ (see for example \cite{LA}, Section 2.4 ).

Furthermore, a tensor $T$ has {\it border rank} $r$ if it is a limit of tensors of
rank $r$ but is not a limit of tensors of rank $s$ for any $s <
r.$ Let $\underline{R}(\T)$ denote the border rank of $T$. Note
that $\underline{R}(\T)\le R(T).$

As in Section \ref{gentriten} we will focus on tri-linear tensors,
we recall here that given a tensor
$T \in V_1 \tensor V_2 \tensor V_3$, where $\dim{V_i} = a_i,$ its rank $R(T)$ can also be realized as the minimal number $p$ of
rank $1$ $a_1 \times a_2$-matrices $S_1, \dots, S_p$ such that each
slice $T_{i,j,\hat{k}}$, for a fixed $\hat{k},$ is a linear combination
of such $S_1, \dots, S_p$ (see for example \cite{er-jo},
Theorem $2.1.2.$).

\subsection{Multiview Geometry}
\label{prelimCV}

For the convenience of the reader, in this Section we recall
standard facts and notation for cameras, centers of projection,
and multiple views in the context of projective reconstruction in computer vision. A
{\it scene} is a set of $N$ points $\{\nbXi \}\in \Pin{k}, i=1,
\dots, N.$ A {\it camera} $P$ is a projection from $\Pin{k}$ onto
$\Pin{h},$ ($h< k),$ from a linear center $C_P.$ The target space
$\Pin{h},$ is called {\it view}. Once homogeneous coordinates have
been chosen in $\Pin{k}$ and $\Pin{h},$ $P$ can be identified with
a $(h+1) \times (k+1)-$ matrix of maximal rank, defined up to a
constant, for which we use the same symbol $P.$ With this
notation, $C_P$ is the right annihilator of $P,$ hence a
$(k-h-1)$-space.  Accordingly, if $\mathbf{X}$ is a point in
$\Pin{k},$ we denote its image in the projection equivalently as
$P(\mathbf{X})$ or $P\cdot \mathbf{X}.$

The rows of $P$ represent linear subspaces of
$\Pin{k}=\mathbb{P}(\mathbb{C}^{k+1})$ defining the center of projection  $C_P$
and can be identified with points of the dual space
${\check{\Pin{k}}}=\mathbb{P}(\check{\mathbb{C}}^{k+1})$, within which
they span a linear space of dimension $h,$
$\Lambda_P=\mathbb{P}(L_P)$, where $L_P$ is a complex vector space
of dimension $h+1.$

The right action of $GL(k+1)$ on $P$ corresponds to a change of
coordinates in $\Pin{k},$ while the left action of $GL(h+1)$ can
be thought of either as a change of coordinates in $L_P$ or in the view.

In the context of multiple view geometry, one considers a set of
multiple images of the same scene, obtained from a set of cameras
$P_j:\Pin{k}\setminus C_{P_j} \to \Pin{h_j}$.

Two different images $P_l(\mathbf{X})$ and $P_m(\mathbf{X})$ of
the same point $\mathbf{X}$ are \textit{corresponding points} and,
more generally, $r$ linear subspaces $\mathcal{S}_j \subset
\Pin{h_j},$ $j=1,\dots, r$ are said to be \textit{corresponding}
if there exists at least one point $\mathbf{X} \in \Pin{k}$ such
that $P_j(\mathbf{X})\in \mathcal{S}_j$ for $j=1,\dots, r.$

\subsection{Grassmann Tensors}
\label{grasstens}
In the context of multiview geometry, Hartley and Schaffalitzky,
\cite{Hart-Schaf}, introduced {\it Grassmann tensors}, which
encode the relations between sets of corresponding subspaces in
the various views. We recall here the basic elements of their
construction.

Consider a set of projections $P_j:\Pin{k}\setminus{C_{P_j}} \to
\Pin{h_j},$ $j = 1,\dots,r,$ $h_j \geq 2$ and a {\it profile},
i.e. a partition $(\alpha_1, \alpha_2, \dots, \alpha_r)$ of $k+1,$
where $1 \leq \alpha_j \leq h_j$ for all $j,$ and $\sum\alpha_j =
k+1.$

Let $\{\mathcal{S}_j\},$ $j=1,\dots,r,$ where $\mathcal{S}_j
\subset \Pin{h_j},$ be a set of general $s_j$-spaces, with
$s_j=h_j-\alpha_j,$ and let $S_j$ be the maximal rank
$(h_j+1)\times (s_j+1)-$matrix whose columns are a basis for
$\mathcal{S}_j$. By definition, if all the $\mathcal{S}_j$ are
corresponding subspaces there exist a point $\mathbf{X} \in
\Pin{k}$ such that $P_j(\mathbf{X})\in \mathcal{S}_j$ for
$j=1,\dots, r.$ In other words there exist $r$ vectors
$\mathbf{v_j} \in \mathbb{C}^{s_j+1}$ $j = 1,\dots,r,$ such that:
\begin{equation}
\label{grasssystem}
\begin{bmatrix}
   P_1 & S_1 & 0 & \dots & 0  \\
   P_2 & 0 & S_2 & \dots & 0 \\
  \vdots & \vdots & \vdots & \vdots & \vdots \\
   P_r & 0 & \dots & 0 & S_r \\
\end{bmatrix}%
\cdot
\begin{bmatrix}
   \mathbf{X}\\
   \mathbf{v_1} \\
  \mathbf{v_2} \\
  \vdots \\
  \mathbf{v_r} \\
\end{bmatrix}
=
\begin{bmatrix}
  0 \\
  0 \\
  \vdots \\
  0 \\
\end{bmatrix}.
\end{equation}

The existence of a non trivial solution
$\{\mathbf{X},\mathbf{v_1},\dots,\mathbf{v_r}\}$ for  system
(\ref{grasssystem}) implies that the system matrix has zero
determinant. This determinant can be thought of as an $r$-linear
form, i.e. a tensor, in the Pl\"{u}cker coordinates of the spaces
$\mathcal{S}_j.$ This tensor is called the {\it Grassmann tensor} $\T,$ and
$\T \in V_1 \tensor V_2\tensor ... \tensor V_r$ where $V_i$ is the
$\binom{h_i+1}{h_i-\alpha_i + 1}$ vector space such that $G(s_i, h_i) \subset \mathbb{P}(V_i).$
More explicitly, the entries of the Grassmann tensor are some of
the Pl\"{u}cker coordinates of the matrix:
\begin{equation}
\left[%
\begin{array}{c|c|c|c}
\label{matricetrasposta_r}
  {P_1}^T & {P_2}^T & \dots & {P_r}^T \\
\end{array}%
\right],
\end{equation}
indeed they are, up to sign, the maximal minors of the matrix
(\ref{matricetrasposta_r}) obtained selecting $\alpha_i$ columns
from ${P_i}^T$, for $i=1, \dots, r.$

It is useful to observe the
effect on a Grassmann tensor and its rank of the actions of $GL(k + 1)$ on the
ambient space and of $GL(h_i + 1)$ on the views. A change of
coordinates in the ambient space, realized by a right action of
$GL(k+1)$ on \brref{matricetrasposta_r} does not alter the tensor,
as all entries are multiplied by the same non-zero constant. On
the other hand, any change of coordinates in a view through left
action of $GL(h_i + 1)$ on the corresponding $P_i^{T}$ does alter
the entries of the tensor, but preserves its rank. Indeed, the
change of coordinates in one of the views induces a linear
invertible transformation on $V_i,$ leaving the rank unchanged, as noted in Section \ref{introranks}.

In the following Sections we deal with the cases of two and three
views, in which the Grassmann tensor turns out to be respectively
a matrix and a three dimensional tensor.

\section{Generalized fundamental matrix} \label{genfundmatrix}

\noindent We consider here the case of two views which gives rise
to the notion of {\it generalized fundamental matrix}, introduced
and studied in \cite{tubbAMPA}. Let us consider two maximal rank
projections $A=[a_{i,j}]$ and $B=[b_{i,j}]$ from $\Pin{k}$ to
$\Pin{h_1}$ and to $\Pin{h_2},$ respectively, where $h_1+h_2 \geq
k+1,$ and where $A$ and $B$ are such that their projection centers
$C_A$ and $C_B$ are in general position so that they do not
intersect. This condition is equivalent to the fact that the
linear span $<L_A,L_B>$ is the whole $\check{\mathbb{C}}^{k+1}.$
The images of the two centers of projection $E^A_B = A(C_B)$ and $
E^B_A=B(C_A)$ are subspaces of dimension $k-h_i-1,$ $i=1,2, $
respectively, of the view spaces, usually called {\it epipoles}.

Following \cite{Hart-Schaf}, we choose a profile
$(\alpha_1,\alpha_2),$ with $\alpha_1+\alpha_2=k+1,$ in order to
obtain the constraints necessary to determine the corresponding
tensor, which, in this case, is a matrix called generalized
fundamental matrix. In the following we make explicit how to place
the minors of (\ref{matricetrasposta_r}) as entries of the
generalized fundamental matrix.

In this case, (\ref{matricetrasposta_r}) becomes
\begin{equation}
\left[%
\begin{array}{c|c}
\label{matricetrasposta_2}
  {A}^T & {B}^T \\
\end{array}%
\right]
\end{equation} and the generalized fundamental matrix
$\mathfrak{F}$ is the $\binom{h_1+1}{h_1-\alpha_1+1}
\times\binom{h_2+1}{h_2-\alpha_2+1}$ matrix, whose entries are
some of the Pl\"{u}cker coordinates of the $k-$space
$\Lambda_{AB} \subset \Pin{h_1+h_2+1},$ spanned by the columns of
the above matrix.

Let $I=(i_1, \dots,i_{s_1+1}),$ $J=(j_1, \dots, j_{s_2+1}),$
$\hat{J}=(h_1+1+j_1, \dots, h_1+1+j_{s_2+1})$ with $1 \leq i_1 <
\dots < i_{s_1+1} \leq h_1+1$ and $1 \leq j_1 < \dots < j_{s_2+1}
\leq h_2+1.$ Denote by $I', \hat{J}'$ the (ordered) sets of
complementary indices  $I'= \{r \in \{1, \dots, h_1+1 \}$ such
that $r \notin I\}$ and $\hat{J}'= \{ s \in \{h_1+2, \dots,
h_1+h_2+2 \} \text{ such that } s \notin \hat{J}\}$. Moreover
denote by $A_I$ and $B_J$ the matrices obtained from $A^T$ and
$B^T$ by deleting columns $i_1, \dots, i_{s_1+1}$ and $j_1, \dots,
j_{s_2+1},$ respectively.

\noindent Then the entries of $\mathfrak{F}$ are:
$F_{I,J}=\epsilon(I,J) \det
\begin{bmatrix}
  A_I &  B_J \\
\end{bmatrix}$ where $\epsilon(I,J)$ is $+1$ or $-1$ according to the parity of
the permutation $(I,\hat{J}, I',\hat{J}'),$ with lexicographical
order of the multi-indices $\{I\}$ for the rows and $\{\hat{J}\}$
for the columns.

In other words, one has $F_{I,J}=q_{I,\hat{J}}(\Lambda_{AB})$,
where $q_{K}(\Lambda)$ denotes the dual-Pl\"{u}cker coordinates
(see, for example, \cite{HP}, Vol.I, book II, pg. 292) of the
space $\Lambda,$ with respect to the multi-index $K$.

In \cite{tubbAMPA} the authors proved the following result:
\begin{theorem}
\label{genfundmatteo} The generalized fundamental matrix
$\mathfrak{F}$ for two projections of maximal rank and whose
centers do not intersect each other, with profile
$(\alpha_1,\alpha_2),$ has rank:
$$\rk{\mathfrak{F}} =\binom{(h_1-\alpha_1+1)+(h_2-\alpha_2+1)}{h_1-\alpha_1+1}.$$
\end{theorem}

The proof given in \cite{tubbAMPA} is obtained associating to the
matrix $\mathfrak{F}$ a rational map $\Phi: G(s_1,h_1)
\dashrightarrow G(k-\alpha_1, h_2)$ whose image is the Schubert
variety $\Omega(E^B_A)$ of the $k-\alpha_1$ spaces containing
$E^B_A,$ and showing that
$\rk{\mathfrak{F}}=dim(<\Omega(E^B_A)>)+1,$ where
$<\Omega(E^B_A)>$ is the projective space spanned by
$\Omega(E^B_A).$

In view of desired generalizations, here we give a straightforward proof of Theorem \ref{genfundmatteo}
based on a suitable choice of coordinates in the
projective spaces involved.

Let $L_A$ and
$L_B$ be the two vector spaces of dimension $h_1+1$ and
$h_2+1$,respectively, spanned by the columns of $A^T$ and $B^T$
and let $\Lambda_A=\mathbb{P}(L_A)$ and
$\Lambda_B=\mathbb{P}(L_B)$. We denote with $i$ the dimension of
$I_{A,B}:= L_A \cap L_B$ which, from Grassmann's formula, turns
out to be $i=h_1+h_2-k+1$. Notice that our assumptions on the
profile ($k+1=\alpha_1 + \alpha_2$) imply that $i > 0$.

One can then choose bases
\begin{alignat*}{2}
\{v_1, \dots, v_i, w_{i+1}, \dots, w_{h_1+1}\}& \ {\rm for }\  L_A,\\
\{v_1, \dots, v_i, w'_{i+1}, \dots, w'_{h_2+1}\}&\  {\rm for }\  L_B,
\end{alignat*}
 such that $\{v_1, \dots, v_i\}$ is a basis for $I_{A,B}$.

Through the left action of $GL(h_1+1)$ and $GL(h_2+1)$ on $A$ and
$B$ respectively, one can then assume that the columns of $A^T$ and $B^T$ are
, respectively,$$[v_1, \dots, v_i, w_{i+1}, \dots, w_{h_1+1}]$$  and
$$[v_1, \dots, v_i, w'_{i+1}, \dots, w'_{h_2+1}].$$

With this assumption, $\{v_1, \dots, v_i, w_{i+1}, \dots,
w_{h_1+1}, w'_{i+1}, \dots, w'_{h_2+1}\}$ is a basis of
$\check{\mathbb{C}}^{k+1},$ and, with the right action of $GL(k+1),$ we
can reduce it to the canonical one $\{e_1, \dots,
e_{k+1}\}.$ With this choice, the matrix (\ref{matricetrasposta_2})
becomes the block matrix
$$\Phi^k_{h_1,h_2}:=\left[%
\begin{array}{cc|cc}
  I_i & \mathbf{0} & I_i & \mathbf{0} \\
  \mathbf{0} & I_{h_1+1-i} & \mathbf{0} & \mathbf{0} \\
  \mathbf{0} & \mathbf{0} & \mathbf{0} & I_{h_2+1-i} \\
\end{array}%
\right]$$ where $I_s$ denotes the $s \times s$ identity matrix and
$\mathbf{0}$ are zero matrices.

The columns of $\Phi^k_{h_1,h_2}$ are denoted by:
$$\tiny \left[%
\begin{array}{cccccc|cccccccccc}
  \underline{a}_1& \dots & \underline{a}_i &\underline{b}_{i+1}& \dots & \underline{b}_{h_1+1}
  &\underline{c}_{h_1+2}& \dots &
   \underline{c}_{h_1+1+i} & \underline{d}_{h_1+2+i}& \dots & \underline{d}_{h_1+h_2+2}
\end{array}%
\right]$$

With this choice of basis, the entries of the fundamental matrix
are the maximal minors of $\Phi^k_{h_1,h_2}$ obtained with
$\alpha_1$ columns chosen among the $\underline{a}_j$ and
$\underline{b}_j$ and $\alpha_2$ columns chosen among the
$\underline{c}_j$ and $\underline{d}_j$. The only non vanishing
entries of the fundamental matrix are hence obtained taking all
the columns $\underline{b}_j$ and $\underline{d}_j$ and choosing
$\alpha_1-(h_1+1-i)$ columns among the $\underline{a}_j$ and the
complementary $\alpha_2-(h_2+1-i)$ among the $\underline{c}_j$. It
follows that the non vanishing entries are as many as the possible
choices of $\alpha_1-(h_1+1-i)$ columns among the first $i$ columns
of $\Phi^k_{h_1,h_2}$. In other words the non zero entries of the
fundamental matrix are:
$$\binom{i}{h_2-\alpha_2+1}=\binom{(h_1-\alpha_1+1)+(h_2-\alpha_2+1)}{h_1-\alpha_1+1}.$$

This number is precisely the rank of the fundamental matrix
since non vanishing entries appear in different rows and columns
of the fundamental matrix.

To clarify the above procedure we consider the following example.

\begin{example}
\label{examplep4p3} Consider  two projections from $\Pin{4}$ to
$\Pin{3}$ with profile $(3,2)$. In this case the matrix
(\ref{matricetrasposta_2}) has dimension $5 \times 8.$ The
subspace $\Lambda_{AB}$ is in  $G(4,7) \subset
\Pin{{\binom{8}{5}}-1},$ and the fundamental matrix $\mathfrak{F}$
turns out to be:
$$\mathfrak{F} =
\begin{bmatrix}
  q_{1,5,6} & q_{1,5,7}& q_{1,5,8} & q_{1,6,7} & q_{1,6,8} & q_{1,7,8} \\
  q_{2,5,6} & q_{2,5,7}& q_{2,5,8} & q_{2,6,7} & q_{2,6,8} & q_{2,7,8} \\
  q_{3,5,6} & q_{3,5,7}& q_{3,5,8} & q_{3,6,7} & q_{3,6,8} & q_{3,7,8} \\
  q_{4,5,6} & q_{4,5,7}& q_{4,5,8} & q_{4,6,7} & q_{4,6,8} & q_{4,7,8} \\
\end{bmatrix}$$

\bigskip

\noindent and the matrix $\Phi^4_{3,3}$ is:
$$\Phi^4_{3,3}= \left[%
\begin{array}{cccc|cccc}
  1 & 0 & 0 & 0 & 1 & 0 & 0 & 0 \\
  0 & 1 & 0 & 0 & 0 & 1 & 0 & 0 \\
  0 & 0 & 1 & 0 & 0 & 0 & 1 & 0 \\
  0 & 0 & 0 & 1 & 0 & 0 & 0 & 0 \\
  0 & 0 & 0 & 0 & 0 & 0 & 0 & 1 \\
\end{array}%
\right]$$

\noindent so that the generalized fundamental matrix, in canonical
form, is the following, from which it is evident that $\rk{\mathfrak{F}} = 3$:

$$\mathfrak{F_C} =
\begin{bmatrix}
  0 & 0& 0 & \pm1 & 0 & 0 \\
  0 & \pm1& 0 & 0 & 0 & 0 \\
  \pm1 & 0& 0 & 0 & 0 &0 \\
  0 & 0& 0 &0 & 0 & 0 \\
\end{bmatrix}.$$
\end{example}

\bigskip
\section{Trifocal Grassmann tensors} \label{gentriten}

Let us now consider three projections $P_1, P_2,$ and $P_3,$ from
$\Pin{k}$ to $\Pin{h_1}$, $\Pin{h_2}$ and to $\Pin{h_3},$
respectively, where $h_1+h_2+h_3 \geq k+1,$ and where $P_1, P_2,$
and $P_3,$ are maximal rank matrices.
%FALSO Grassmann formula shows that our assumption
%$h_1+h_2+h_3 \geq k+1,$ in particular, implies that, for generic
%choices of $A, B,$ and $C,$ their projection centers  $C_A, C_B,$
%and $C_C$ are mutually disjoint.FALSO

Grassmann's formula shows that for generic choices of $P_1, P_2,$
and $P_3,$ their projection centers  $C_1, C_2,$ and $C_3$ are
mutually disjoint under the assumptions: $k-h_i+h_j-1 \leq 0,$ for
$1 \leq i,j \leq 3, i \neq j.$

As in the case of the generalized fundamental matrix, let
$(\alpha_1,\alpha_2,\alpha_3),$ be a profile with $\alpha_1+\alpha_2+\alpha_3=k+1,$ in
order to obtain the necessary constraints to determine the
corresponding tensor. The tensor thus obtained is called the {\it
trifocal Grassman tensor} and it is a generalization of the
classical trifocal tensor for three views in $\Pin{3}.$
Its entries can be explicitly computed from (\ref{grasssystem}), as shown below.

In this case, (\ref{matricetrasposta_r}) becomes
\begin{equation}
\left[%
\begin{array}{c|c|c}
\label{matricetrasposta_3}
  {P_1}^T & {P_2}^T & {P_3}^T\\
\end{array}%
\right]
\end{equation} and the entries of the trifocal tensor
$\mathcal{T}$ are, up to sign,  some of the maximal minors of the matrix
(\ref{matricetrasposta_3}) obtained by choosing $\alpha_1$ columns
in ${P_1}^T$, $\alpha_2$ in ${P_2}^T$ and $\alpha_3$ in ${P_3}^T$.

More explicitly, let $I=(i_1, \dots, i_{s_1+1}),$ $J=(j_1, \dots,
j_{s_2+1}),$ $K=(k_1, \dots, k_{s_3+1}),$ $\hat{J}=(h_1+1+j_1,
\dots, h_1+1+j_{s_2+1})$ and $\hat{K}=(h_1+h_2+2+k_1, \dots,
h_1+h_2+2+k_{s_3+1})$ with $1 \leq i_1 < \dots < i_{s_1+1} \leq
h_1+1$, $1 \leq j_1 < \dots < j_{s_2+1} \leq h_2+1$ and $1 \leq
k_1 < \dots < k_{s_3+1} \leq h_3+1$.

Denote by $I', \hat{J}',\hat{K}'$ the (ordered) sets of
complementary indices $I'= \{r \in \{1, \dots, h_1+1 \}$ such that
$r \notin I\}$ and $\hat{J}'= \{ s \in \{h_1+2, \dots, h_1+h_2+2 \}
\text{ such that } s \notin \hat{J}\}$ and $\hat{K}'= \{ t \in
\{h_1+h_2+3, \dots, h_1+h_2+h_3+3 \} \text{ such that } t \notin
\hat{K}\}$. Moreover denote by ${P_1}_I$, ${P_2}_J$ and ${P_3}_K$
respectively, the matrices obtained from ${P_1}^T$, ${P_2}^T$ and
${P_3}^T$ deleting columns $i_1, \dots, i_{s_1+1}$, $j_1, \dots,
j_{s_2+1}$ and $k_1, \dots, k_{s_3+1},$ respectively.
Let $\epsilon(i_1,\dots, i_n)$ be $+1$ or $-1$ according to the parity of the permutation $(i_1,\dots, i_n).$
 The entries of $\mathcal{T}$ are given by:
\begin{equation}
\label{fijk}
\T_{I,J,K}=\epsilon(I,\hat{J},\hat{K},
I',\hat{J}',\hat{K}') \det
\begin{bmatrix}
  {P_1}_I \\
  {P_2}_J \\
  {P_3}_K \\
\end{bmatrix} \end{equation}.

Denote by $V_i$ the vector space such that $G(s_i,h_i) \subseteq
\Pin{\binom{h_i+1}{s_i+1}-1}= \mathbb{P}(V_i)$. The {\it
trifocal Grassmann tensor} for three projections $P_1,P_2,P_3$
from $\Pin{k}$ to $\Pin{h_1}$, $\Pin{h_2}$ and $\Pin{h_3},$ with
profile $(\alpha_1, \alpha_2, \alpha_3),$ is, up to a
multiplicative non zero constant,  the
$\binom{h_1+1}{h_1-\alpha_1+1}
\times\binom{h_2+1}{h_2-\alpha_2+1}\times\binom{h_3+1}{h_3-\alpha_3+1}$
tensor $\mathcal{T}\in V_1 \otimes V_2 \otimes V_3,$ whose
entries are $\T_{I,J,K}$ with lexicographical order of the families
$\{I\}, \{J\},$ and $\{K\}$ of multi-indices.

\section{The Rank of trifocal Grassmann tensors}

In the classical case of projections from $\Pin{3}$ to $\Pin{2}$,
the rank of the trifocal tensor is known to be $4,$ (e.g. see
\cite{oe1}, \cite{hey1}), while the rank of the quadrifocal tensor
turns out to be $9,$ \cite{hey1}. Nothing further is known in general
about the ranks of Grassmann tensors. In this Section
first we provide a canonical form for the matrix
(\ref{matricetrasposta_3}), in analogy to what was done for the
two views case. Then, using this canonical form, we compute
$R(\mathcal{T})$ in the general case, i.e. when the center of
projections are in general position (see Assumption \ref{g.a.}). The non general cases are discussed in
Section \ref{uprank}.

\subsection{Canonical form} \label{canonical}

Let $L_1$,
$L_2$ and $L_3$ be the vector spaces of dimension $h_1+1$, $h_2+1$
and $h_3+1$ respectively, spanned by the columns of ${P_1}^T$,
${P_2}^T$ and ${P_3}^T$ and let $\Lambda_1=\mathbb{P}(L_1)$,
$\Lambda_2=\mathbb{P}(L_2)$ and $\Lambda_3=\mathbb{P}(L_3)$.

We consider, for each triplet of distinct integers $r,s,t=1,2,3,$
the following integers:

\begin{align}
&i_{r,s}=h_r+h_s+1-k;\label{numrel1}\\
&i=h_1+h_2+h_3+1-2k;\label{numrel2} \\
&j_{r,s}=i_{r,s}-i=k-h_t.\label{numrel3}
\end{align}
\bigskip
\noindent Our generality assumption is the following:
\begin{assumption}

%\item [a)] each pair of vector spaces $L_r$ and $L_s$ span
%$\mathbb{C}^{k+1},$ or, equivalently, each pair of centers of
%projection do not intersect, and

 \label{g.a.} For any choice of $r,s,t$ with $\{r,s,t\}=\{1,2,3\}$, the intersection $\Lambda_{rs}=L_r \cap L_s$ with
$L_t$ span $\mathbb{C}^{k+1},$ or, equivalently, the span of each
pair of centers do not intersect the third one.
\end{assumption}

\noindent This assumption implies, in particular, that for any
choice of a pair $r,s$, the span of $L_r$ and $L_s$ is the whole
$\mathbb{C}^{k+1},$ or, in other words, that the two centers $C_r$
and $C_s$ do not intersect.

Under Assumption \ref{g.a.}, applying Grassmann formula one sees that the three
numbers above have the following meaning: $i_{r,s}= dim (L_r \cap
L_s) \geq 0$, for any choice of $r,s$ , $i= dim (L_1 \cap L_2 \cap
L_3) \geq 0$ and $j_{r,s}$ is the affine dimension of the center
$C_t$ i.e. $k-{h_t}=j_{rs}$ for $r,s,t=1,2,3.$

Hence we can choose bases as follows:

$$L_1 \cap L_2 \cap L_3 = <v_1, \dots, v_i>$$

$$L_1 \cap L_2 = <v_1, \dots, v_i, w_{1}, \dots, w_{j_{1,2}}>$$

$$L_1 \cap L_3 = <v_1, \dots, v_i,  u_{1}, \dots, u_{j_{1,3}}>$$

$$L_2 \cap L_3 = <v_1, \dots, v_i, s_{1}, \dots, s_{j_{2,3}}>$$

\

\noindent so that:

$$L_1 = <v_1, \dots, v_i, w_{1}, \dots, w_{j_{1,2}}, u_{1}, \dots, u_{j_{1,3}}>,$$

$$L_2 = <v_1, \dots, v_i, w_{1}, \dots, w_{j_{1,2}}, s_{1}, \dots, s_{j_{2,3}}>,$$

$$L_3 = <v_1, \dots, v_i, u_{1}, \dots, u_{j_{1,3}}, s_{1}, \dots, s_{j_{2,3}}>.$$

\

\noindent Through the left action of $GL(h_i+1)$ on $P_i,$
$i=1,2,3,$ one can assume that the columns of ${P_1}^T, {P_2}^T,$ and ${P_3}^T$ are, respectively:
$$[v_1, \dots, v_i, w_{1}, \dots, w_{j_{1,2}}, u_{1}, \dots, u_{j_{1,3}}],$$
$$[v_1, \dots, v_i, w_{1}, \dots, w_{j_{1,2}}, s_{1}, \dots, s_{j_{2,3}}],$$
$$[v_1, \dots, v_i, u_{1}, \dots, u_{j_{1,3}}, s_{1}, \dots, s_{j_{2,3}}].$$

\noindent With this assumption,
\begin{equation}
\label{basebuona}
\{v_1, \dots, v_i, w_{1},
\dots, w_{j_{1,2}}, u_{1}, \dots, u_{j_{1,3}},s_{1}, \dots,
s_{j_{2,3}}\}
\end{equation} is a basis of $\check{\mathbb{C}}^{k+1}.$

\

\noindent With the right action of $GL(k+1)$ we can reduce \brref{basebuona} to the canonical basis.

\noindent With this choice, the matrix (\ref{matricetrasposta_3}) becomes the block matrix:
\begin{equation}
\label{canmat}
\Phi^k_{h_1,h_2,h_3}:=\left[%
\begin{array}{ccc|ccc|ccc}
  I_i & \mathbf{0} & \mathbf{0} & I_i & \mathbf{0} & \mathbf{0} & I_i & \mathbf{0} & \mathbf{0} \\
  \mathbf{0} & I_{j_{1,2}} & \mathbf{0} & \mathbf{0} & I_{j_{1,2}} & \mathbf{0} & \mathbf{0} & \mathbf{0} & \mathbf{0}\\
  \mathbf{0} & \mathbf{0} & I_{j_{1,3}} & \mathbf{0} & \mathbf{0} & \mathbf{0} & \mathbf{0} & I_{j_{1,3}} & \mathbf{0}\\
  \mathbf{0} & \mathbf{0} & \mathbf{0} & \mathbf{0} & \mathbf{0}& I_{j_{2,3}} & \mathbf{0} & \mathbf{0} & I_{j_{2,3}}\\
\end{array}%
\right].
\end{equation}

\subsection{The rank} \label{rankgp}
The canonical form $\Phi^k_{h_1,h_2,h_3}$ of matrix \brref{matricetrasposta_3} allows one to successfully
compute the rank of trifocal Grassmann tensors.
\begin{theorem}
\label{therank} Let  $P_l:\Pin{k} \to \Pin{h_l},$ $l=1,2,3,$ be maximal rank projections whose centers satisfy
Assumption \ref{g.a.}.
The trifocal Grassmann tensor $\mathcal{T}$
for projections $\{P_l\},$ with
profile $(\alpha_1,\alpha_2,\alpha_3),$ has rank:
\begin{equation}\tiny
\sum_{a_2=0}^{j_{12}}\sum_{a_3=0}^{j_{13}}\sum_{b_3=0}^{j_{23}}\binom{j_{12}}{a_2}\binom{j_{13}}{a_3}\binom{j_{23}}{b_3}
\binom{i}{\alpha_1-a_2-a_3}\binom{i-\alpha_1+a_2+a_3}{\alpha_2-j_{12}+a_2-b_3},
\end{equation}
where $i= h_1+h_2+ h_3 + 1 -2k$ and $j_{rs}= k - h_t$ for $\{r,s,t\} = \{ 1,2,3\}.$
\end{theorem}

\begin{proof}

Let $\Phi^k_{h_1,h_2,h_3}$ be the canonical form of matrix \brref{matricetrasposta_3} associated to
the given projections $P_l:\Pin{k} \to \Pin{h_l},$ $l=1,2,3,$ and let $[\Phi^k_{h_1,h_2,h_3}]_r^s$
denote the submatrix of $\Phi^k_{h_1,h_2,h_3}$ consisting of
consecutive columns from column $r,$ included, to column $s,$ included.
To generate each entry of the tensor $\mathcal{T}$ one must choose:
\begin{itemize}
\item[-] $a_1$ columns from $\Phicols{1}{i},$ \\

\item[-] $a_2$ columns from $\Phicols{{i+1}}{{i+j_{12}}},$\\

\item[-] $a_3$ columns from $\Phicols{{i+j_{12}+1}}{{i+j_{12}+j_{13}}},$
\end{itemize} with $a_1+a_2+a_3=\alpha_1$. \\
Similarly one has to choose:
\begin{itemize}
\item[-] $b_1$ columns from $\Phicols{{i+j_{12}+j_{13}+1}}{{2i+j_{12}+j_{13}}},$\\

\item[-] $b_2$ columns from $\Phicols{{2i+j_{12}+j_{13}+1}}{{2i+2j_{12}+j_{13}}},$\\

\item[-] $b_3$ columns from $\Phicols{{2i+2j_{12}+j_{13}+1}}{{2i+2j_{12}+j_{13}+j_{23}}},$
\end{itemize}
with $b_1+b_2+b_3=\alpha_2$. \\
Finally one has to choose:
\begin{itemize}
\item[-] $c_1$ columns from $\Phicols{{2i+2j_{12}+j_{13}+j_{23}+1}}{{3i+2j_{12}+j_{13}+j_{23}}},$\\

\item[-] $c_2$ columns from $\Phicols{{3i+2j_{12}+j_{13}+j_{23}+1}}{{3i+2j_{12}+2j_{13}+j_{23}}},$\\

\item[-] $c_3$ columns from $\Phicols{{3i+2j_{12}+2j_{13}+j_{23}+1}}{{3i+2j_{12}+2j_{13}+2j_{23}}},$
\end{itemize}
with $c_1+c_2+c_3=\alpha_c$.

Moreover to get non vanishing entries of $\T$, the following equalities
must be satisfied:
\begin{itemize}
\item $a_1+b_1+c_1=i$

\item $a_2+b_2=j_{12}$

\item $a_3+c_2=j_{13}$

\item $b_3+c_3=j_{23}$.
\end{itemize}

From the above conditions, the number of non vanishing entries of
the tensor is given by:

 \begin{equation}\tiny \label{ubipos2}
\sum_{a_2=0}^{j_{12}}\sum_{a_3=0}^{j_{13}}\sum_{b_3=0}^{j_{23}}\binom{j_{12}}{a_2}\binom{j_{13}}{a_3}\binom{j_{23}}{b_3}
\binom{i}{\alpha_1-a_2-a_3}\binom{i-\alpha_1+a_2+a_3}{\alpha_2-j_{12}+a_2-b_3}.
\end{equation}
Clearly (\ref{ubipos2}) gives an upper bound for $R(\mathcal{T}).$
To prove that (\ref{ubipos2}) is equal to $R(\mathcal{T}),$ we
use the slices-based characterization of the rank recalled at the end of  Section \ref{introranks}.

\noindent In our case the positions of the non zero entries of
$\mathcal{T}$ are different for different faces, i.e. if
$T_{\bar{I},\bar{J},\bar{K}} \neq 0$, the
$T_{\bar{I},\bar{J},K}=0$ for all $K \neq \bar{K}$. The reason is
that once the columns determined by the multi-indexes $I$ and $J$
are chosen there is at most one possible choice of the columns
determined by $K$ which gives a non vanishing minor.

\noindent This completes the proof.
\end{proof}

The above result is further illustrated by the two following explicit examples.

\begin{example}
\label{examplep3p2p2p2} In the case of the classical  $3 \times 3
\times 3$ trifocal tensor, i.e. of three projections from
$\Pin{3}$ to $\Pin{2}$ with profile $(2,1,1)$, we get: $i=1$ and
$i_{rs}=2$ for each $r,s$. Hence, in this case, \brref{canmat} is:
$$\Phi^3_{2,2,2}:=\left[%
\begin{array}{ccc|ccc|ccc}
  1 & 0 & 0 & 1 & 0 & 0 & 1 & 0 & 0 \\
  0 & 1 & 0 & 0 & 1 & 0 & 0 & 0 & 0\\
  0 & 0 & 1 & 0 & 0 & 0 & 0 & 1 & 0\\
  0 & 0 & 0 & 0 & 0 & 1 & 0 & 0 & 1\\
\end{array}%
\right].$$
The only non vanishing elements of the tensor are: $\T_{131},
\T_{113},\T_{221},\T_{312}$, hence $R(\T) = 4.$
\end{example}

\begin{example}
\label{examplep4p3p3p2} In the case of three projections from
$\Pin{4}$ to $\Pin{3}$, $\Pin{3}$ and $\Pin{2},$ with profile
$(2,2,1)$, we get: $i=1, i_{12}=3$, and $i_{13}=i_{23}=2$. Hence, in this
case, \brref{canmat} becomes:
$$\Phi^4_{3,3,2}:=\left[%
\begin{array}{cccc|cccc|ccc}
  1 & 0 & 0 & 0 & 1 & 0 & 0 & 0 & 1 & 0 & 0\\
  0 & 1 & 0 & 0 & 0 & 1 & 0 & 0 & 0 & 0 & 0\\
  0 & 0 & 1 & 0 & 0 & 0 & 1 & 0 & 0 & 0 & 0\\
  0 & 0 & 0 & 1 & 0 & 0 & 0 & 0 & 0 & 1 & 0\\
  0 & 0 & 0 & 0 & 0 & 0 & 0 & 1 & 0 & 0 & 1\\
\end{array}%
\right].$$

The trifocal tensor $\T$ is a $6 \times 6 \times 3$ tensor and  its non
vanishing elements are: $\T_{123}, \T_{161},\T_{213},\T_{251},\T_{341},
\T_{431},\T_{522},\T_{612}$, hence $R(\T) = 8.$

Moreover, one sees that $\T$ is a linear combination of :
\begin{alignat*}{4}
\mathbf{e^1_1}\otimes \mathbf{e^2_2}\otimes \mathbf{e^3_3},\quad &
\mathbf{e^1_1}\otimes \mathbf{e^2_6}\otimes\mathbf{e^3_1},\quad &
\mathbf{e^1_2}\otimes \mathbf{e^2_1}\otimes \mathbf{e^3_3},\quad &
\mathbf{e^1_2}\otimes \mathbf{e^2_5} \otimes \mathbf{e^3_1},\\
\mathbf{e^1_3}\otimes \mathbf{e^2_4} \otimes \mathbf{e^3_1},\quad &
\mathbf{e^1_4}\otimes \mathbf{e^2_3} \otimes \mathbf{e^3_1},\quad &
\mathbf{e^1_5}\otimes \mathbf{e^2_2} \otimes \mathbf{e^3_2},\quad &
\mathbf{e^1_6}\otimes \mathbf{e^2_1} \otimes \mathbf{e^3_2},
\end{alignat*}

where $\mathbf{e^r_s}$ is the $s$-element of the canonical base of the
vector space $V_r = \mathbb{C}^{\binom{h_r+1}{h_r-\alpha_r+1}}.$

\end{example}

\section{The non general case}
\label{uprank}

In this Section we consider cases in which Assumption \ref{g.a.}
is not satisfied, and the rank
depends on the degenerate geometric configurations of the projections. This is
evident also in the simplest case of the classical trifocal tensor
for which the rank is $4$ for general projections (example
\ref{examplep3p2p2p2}) and becomes $5$ when the three centers are on a line (example \ref{examplep3p2p2p2deg}).

If Assumption \ref{g.a.} is not satisfied one can no longer obtain
canonical form \brref{canmat} for the combined projection matrices, because the
integers defined in \brref{numrel1}, \brref{numrel2}, and \brref{numrel3} lose their geometric meaning
and, moreover, \brref{numrel3} may no longer hold.

In this situation one can obtain a different canonical form,
from which the rank of the Grassmann tensor can be
computed in concrete cases.

We introduce the following notations:
\begin{itemize}

\item $g:=\dim{(L_1 \cap L_2 \cap L_3)}$;

\item $g_{rs}:=\dim{(L_r \cap L_s)}$;

\item $l_{rs}:=g_{rs}-g$;

\item $\alpha_{rs}$ the non negative integer such that the span
$<L_r,L_s>$ has dimension $k+1-\alpha_{rs}$;

\item $\beta_{rs}$ the non negative integer such that the span
$<\Lambda_{rs},L_t>$ has dimension $k+1-\beta_{rs}$.
\end{itemize}

\noindent By Grassmann formula, these integers are linked to the
ones in \brref{numrel1}, \brref{numrel2}, and \brref{numrel3} as follows: $g=i+\alpha_{rs}+\beta_{rs}$ and
$g_{rs}=i_{rs}+\alpha_{rs}$ for any $r,s$.

Arguing as in the previous Section, where $g$
and $l_{rs}$ now play the role of $i$ and $j_{rs}$ respectively, by choosing the first $g+l_{12}+l_{13}+l_{23}$ vectors of the canonical base of
$\mathbb{C}^{k+1}$
one gets the following canonical form for the matrix
(\ref{matricetrasposta_3}), which now depends also on
$\alpha_{rs}$ and $\beta_{rs}.$

$$\tiny{\Psi^k_{h_1,h_2,h_3}:=\left[%
\begin{array}{cccc|cccc|cccc}
  I_g & \mathbf{0} & \mathbf{0}& Z^1_{1} & I_g & \mathbf{0} & \mathbf{0}& Z^1_{2} & I_g & \mathbf{0} & \mathbf{0}
  & Z^1_{3}\\
  \mathbf{0} & I_{l_{1,2}} & \mathbf{0} & Z^2_{1} & \mathbf{0} & I_{l_{1,2}} & \mathbf{0} & Z^2_{2} & \mathbf{0}
  & \mathbf{0} & \mathbf{0}& Z^2_{3}\\
  \mathbf{0} & \mathbf{0} & I_{l_{1,3}} & Z^3_{1} & \mathbf{0} & \mathbf{0} & \mathbf{0} & Z^3_{2} & \mathbf{0}
  & I_{l_{1,3}} & \mathbf{0}& Z^3_{3}\\
  \mathbf{0} & \mathbf{0} & \mathbf{0}& Z^4_{1} & \mathbf{0} & \mathbf{0}& I_{l_{2,3}} & Z^4_{2} & \mathbf{0}
  & \mathbf{0} & I_{l_{2,3}}& Z^4_{3}\\
  \mathbf{0} & \mathbf{0} & \mathbf{0}& Z^5_{1} & \mathbf{0} & \mathbf{0}& \mathbf{0} & Z^5_{2} & \mathbf{0}
  & \mathbf{0} & \mathbf{0}& Z^5_{3}\\
\end{array}%
\right].}$$

\noindent In the matrix $\Psi^k_{h_1,h_2,h_3}$, the submatrices $Z^p_t$, with
$t=1,2,3$ and $p=1,2,3,4$ have $(h_t+1-g-l_{rt}-l_{st})$ columns.
Moreover, by an iterated use of Grassmann formula, one sees that
$k+1-g-l_{12}-l_{13}-l_{23}=2(\alpha_{rs}+\beta_{rs})-(\alpha_{12}+\alpha_{13}+\alpha_{23})$
so that the matrices $Z^5_t$ have
$2(\alpha_{rs}+\beta_{rs})-(\alpha_{12}+\alpha_{13}+\alpha_{23})$
rows.

\noindent Suitable left actions of $GL(h_i+1)$ on the views give the following form for $\Psi^k_{h_1,h_2,h_3}:$

$$\tiny \Psi^k_{h_1,h_2,h_3}:=\left[%
\begin{array}{cccc|cccc|cccc}
  I_g & \mathbf{0} & \mathbf{0}& \mathbf{0} & I_g & \mathbf{0} & \mathbf{0}& \mathbf{0} & I_g & \mathbf{0} & \mathbf{0}
  & \mathbf{0}\\
  \mathbf{0} & I_{l_{1,2}} & \mathbf{0} & \mathbf{0} & \mathbf{0} & I_{l_{1,2}} & \mathbf{0} & \mathbf{0} & \mathbf{0}
  & \mathbf{0} & \mathbf{0}& Z^2_{3}\\
  \mathbf{0} & \mathbf{0} & I_{l_{1,3}} & \mathbf{0} & \mathbf{0} & \mathbf{0} & \mathbf{0} & Z^3_{2} & \mathbf{0}
  & I_{l_{1,3}} & \mathbf{0}& \mathbf{0}\\
  \mathbf{0} & \mathbf{0} & \mathbf{0}& Z^4_{1} & \mathbf{0} & \mathbf{0}& I_{l_{2,3}} & \mathbf{0} & \mathbf{0}
  & \mathbf{0} & I_{l_{2,3}}& \mathbf{0}\\
  \mathbf{0} & \mathbf{0} & \mathbf{0}& Z^5_{1} & \mathbf{0} & \mathbf{0}& \mathbf{0} & Z^5_{2} & \mathbf{0}
  & \mathbf{0} & \mathbf{0}& Z^5_{3}\\
\end{array}%
\right].$$

 The following
examples illustrate how, depending on $h_t$, the form of $\Psi^k_{h_1,h_2,h_3}$ can be further
simplified by choosing additional vectors in the canonical basis of
$\mathbb{C}^{k+1},$ as columns of the matrices $Z^p_t$.

Moreover it is clear that the rank of the Grassmann tensor $R(\T)$ depends
on the entries of the matrices $Z^p_t$, hence an
explicit formula for $R(\T)$ is not provided.
Nevertheless, as shown in
the examples below, in specific concrete cases the number of non vanishing
elements of the tensor can be computed, and thus  an upper bound for $R(\T)$ can be obtained.

%Now we can argue similarly to section \ref{rankgp} obtaining the
%following number of possibly non zero entries:

%{\tiny{\begin{equation} \label{ubineg}
%\sum_{a_1=0}^{i_{12}}\sum_{a_2=0}^{i_{13}}\sum_{b_1=0}^{i_{12}}\sum_{b_2=0}^{i_{23}}\binom{i_{12}}{a_1}\binom{i_{13}}{a_2}
%\binom{g}{\alpha_1-a_1-a_2}\\
%\binom{i_{12}-1}{b_1}\binom{i_{23}}{b_2}\binom{g}{\alpha_2-b_1-b_2}
%  \binom{g}{\alpha_3-i_{13}+a_2-i_{23}+b_2}.
%\end{equation}}}
%This number gives an upper bound for $R(\mathcal{T}).$ Also here

\begin{example}
\label{examplep5p2p2p2} In the case of three projections from
$\Pin{5}$ to $\Pin{2}$, $\Pin{2}$ and $\Pin{2},$ with profile
$(2,2,2)$, we get: $g=g_{rs}=l_{rs}=0$,  $\alpha_{rs}=0$ and
$\beta_{rs}=3$, for each $r,s$. In this case $\Psi^5_{2,2,2}$ reduces to
$[Z^5_1|Z^5_2|Z^5_3]$, where each $Z^5_t$ is a $(6 \times 3)$
matrix. Up to now we have not yet fixed any vector of the basis,
so that, with a further choice of the reference frame, we get:

$$\Psi^5_{2,2,2}:=\left[%
\begin{array}{ccc|ccc|ccc}
  1 & 0 & 0 & 0 & 0 & 0 & z_{11} & z_{12} & z_{13} \\
  0 & 1 & 0 & 0 & 0 & 0 & z_{21} & z_{22} & z_{23} \\
  0 & 0 & 1 & 0 & 0 & 0 & z_{31} & z_{32} & z_{33} \\
  0 & 0 & 0 & 1 & 0 & 0 & z_{41} & z_{42} & z_{43} \\
  0 & 0 & 0 & 0 & 1 & 0 & z_{51} & z_{52} & z_{53} \\
  0 & 0 & 0 & 0 & 0 & 1 & z_{61} & z_{62} & z_{63} \\
\end{array}%
\right].$$

The trifocal tensor $\T$ is a $3 \times 3 \times 3$ tensor and for
generic choices of $z_{ij}$, all its elements are non vanishing and thus no significant upper bound for the rank can be given.
\end{example}

The following example is a degenerate configuration of the
classical trifocal tensor.
\begin{example}
\label{examplep3p2p2p2deg} In the case of three projections from
$\Pin{3}$ to $\Pin{2}$ with profile $(2,1,1)$ and
centers of projection on a line, one has: $g=g_{rs}=2,$ $ l_{rs}=0$,  $\alpha_{rs}=0$ and
$\beta_{rs}=1$, for each $r,s$. In this case $\Psi^3_{2,2,2}$ reduces to
$$\Psi^3_{2,2,2}:=\left[%
\begin{array}{ccc|ccc|ccc}
  1 & 0 & z_{11} & 1 & 0 & z_{12} & 1 & 0 & z_{13} \\
  0 & 1 & z_{21} & 0 & 1 & z_{22} & 0 & 1 & z_{23}\\
  0 & 0 & z_{31} & 0 & 0 & z_{32} & 0 & 0 & z_{33}\\
  0 & 0 & z_{41} & 0 & 0 & z_{42} & 0 & 0 & z_{43}\\
\end{array}%
\right].$$ Further changes of coordinates, both
in the ambient space and in the views, gives:
$$\Psi^3_{2,2,2}:=\left[%
\begin{array}{ccc|ccc|ccc}
  1 & 0 & 0 & 1 & 0 & 0 & 1 & 0 & 0 \\
  0 & 1 & 0 & 0 & 1 & 0 & 0 & 1 & 0\\
  0 & 0 & 1 & 0 & 0 & 0 & 0 & 0 & a\\
  0 & 0 & 0 & 0 & 0 & 1 & 0 & 0 & b\\
\end{array}%
\right],$$ with $a$ and $b \neq 0.$

The only non vanishing elements of the tensor are: $$T_{113},
T_{131},T_{212},T_{221},T_{311},$$ hence $R(\T) = 5,$ while the
rank of the classical general trifocal is $4.$
\end{example}

\subsection{Border ranks}
\label{brank} Examples (\ref{examplep3p2p2p2}) and
(\ref{examplep3p2p2p2deg}) seen above, provide evidence, already in the classical setting of projective reconstruction in $\Pin{3},$ of the fact that the rank of tensors is not semicontinuous.

Indeed, it is very easy to construct a
one dimensional family of triplets of point (centers of
projection) which do not lie on a line but converge to a triplet of points on a line.
In other words a family of rank $4$ tensors which
converges to a rank $5$ one.

The general situation is still more intricate: even in the first non classical cases of
$\Pin{4}$ as ambient spaces, we provide some topical examples
which display the breadth of phenomena that can occur.

\begin{example}
\label{examplep4p2p2p2} In the case of three projections from
$\Pin{4}$ to $\Pin{2}$, $\Pin{2}$ and $\Pin{2},$ with profile
$(2,2,1)$, Assumption \ref{g.a.} doesn't hold, and we get: $g=0,
g_{rs}=l_{rs}=1$,  $\alpha_{rs}=0$ and $\beta_{rs}=1$, for each
$r,s$. In this case $\Psi^4_{2,2,2}$ reduces to
$$\Psi^4_{2,2,2}:=\left[%
\begin{array}{ccc|ccc|ccc}
  1 & 0 &  0 & 1 & 0 &  0 & 0 & 0 &  z_{13} \\
  0 & 1 &  0 & 0 & 0 &  z_{22} & 1 & 0 &  0 \\
  0 & 0 &  z_{31} & 0 & 1 &  0 & 0 & 1 &  0 \\
  0 & 0 &  z_{41} & 0 & 0 &  z_{42} & 0 & 0 &  z_{43} \\
  0 & 0 &  z_{51} & 0 & 0 &  z_{52} & 0 & 0 &  z_{53} \\
\end{array}%
\right].$$ Again, a further change of coordinates in
the ambient space, gives:

$$\Psi^4_{2,2,2}:=\left[%
\begin{array}{ccc|ccc|ccc}
  1 & 0 & 0 & 1 & 0 & 0 & 0 & 0 & a \\
  0 & 1 & 0 & 0 & 0 & 0 & 1 & 0 & 0 \\
  0 & 0 & 0 & 0 & 1 & 0 & 0 & 1 & 0 \\
  0 & 0 & 1 & 0 & 0 & 0 & 0 & 0 & b \\
  0 & 0 & 0 & 0 & 0 & 1 & 0 & 0 & c \\
\end{array}%
\right],$$

\noindent with $a,b,c \neq 0$.

The trifocal tensor $\T$ is a $3\times 3 \times 3$ tensor and  its non
vanishing elements are: $\T_{111},
\T_{122},\T_{131},\T_{213},\T_{311}$, from which one easily deduce
that $R(\T) = 4,$ because the tensor is a linear combination
of:
$$(\mathbf{e^1_1}+\mathbf{e^1_3})\otimes
\mathbf{e^2_1}\otimes \mathbf{e^3_1},\quad \mathbf{e^1_1}\otimes
\mathbf{e^2_2}\otimes \mathbf{e^3_2},\quad \mathbf{e^1_1}\otimes
\mathbf{e^2_3}\otimes \mathbf{e^3_1},\quad \mathbf{e^1_2}\otimes
\mathbf{e^2_1}\otimes \mathbf{e^3_3}.$$

Starting from the above example, one can consider the following
degenerate configurations for lines $C_A, C_B, C_C,$ which are
centers of projection. Notice that each of these configurations can easily obtained as a limit of a sequence of non degenerate configurations of centers of projection.

\begin{itemize}
\item[a)] $C_A,C_B, C_C$ lie in the same hyperplane and no two of them intersect each other;
\item[b)] $C_A,C_B, C_C$ span $\Pin{4}$ but two of them have nonempty intersection;
\item[c)] $C_A,C_B, C_C$ lie in the same hyperplane and two of them have nonempty intersection.
\end{itemize}

With suitable choices of coordinates and similarly to the rank
calculations performed above, one sees that, respectively:
\begin{itemize}
\item[a)] $g=g_{rs}, l_{rs}=0$,\\
 $\alpha_{rs}=0$ and $\beta_{rs}=2$, for each $r,s$. \\
 In this case $\Psi^4_{2,2,2}$ reduces to
$$\Psi^4_{2,2,2}:=\left[%
\begin{array}{ccc|ccc|ccc}
  1 & 0 & 0 & 1 & 0 & 0 & 1 & 0 & 0 \\
  0 & 1 & 0 & 0 & 0 & 0 & 0 & z_{11} & z_{12} \\
  0 & 0 & 1 & 0 & 0 & 0 & 0 & z_{21} & z_{22} \\
  0 & 0 & 0 & 0 & 1 & 0 & 0 & z_{31} & z_{32} \\
  0 & 0 & 0 & 0 & 0 & 1 & 0 & z_{41} & z_{42} \\
\end{array}%
\right].$$

\noindent The non vanishing elements of the tensor are:
$$\T_{113}, \T_{121},\T_{122},\T_{131},\T_{132},
\T_{211},\T_{212},\T_{311},\T_{312}$$ and $R(\T)$ jumps to $5$. With the same
notation of example \ref{examplep4p3p3p2}, one sees that $\T$ is a
combination of:
\begin{alignat*}{3}
\mathbf{e^1_1}\otimes \mathbf{e^2_1}\otimes \mathbf{e^3_3},\quad
&\mathbf{e^1_1}\otimes \mathbf{e^2_2}\otimes (\mathbf{e^3_1}+ \mathbf{e^3_2}),\quad
&\mathbf{e^1_1}\otimes \mathbf{e^2_3}\otimes \mathbf{e^3_1},\\
\mathbf{e^1_2}\otimes \mathbf{e^2_1}\otimes (\mathbf{e^3_1}+ \mathbf{e^3_2}),\quad
&\mathbf{e^1_3}\otimes \mathbf{e^2_1}\otimes (\mathbf{e^3_1} +\mathbf{e^3_2}) . \quad
&
\end{alignat*}

\item[b)]$g=0, g_{12}=l_{12}=2$, $g_{13}=g_{23}=l_{13}=l_{23}=0,$\\
$\alpha_{12}=2, \beta_{12}=0,
\alpha_{13}=\alpha_{23}=0,\beta_{13}=\beta_{23}=2$. \\
In this case $\Psi^4_{2,2,2}$ reduces to

$$\Psi^4_{2,2,2}:=\left[%
\begin{array}{ccc|ccc|ccc}
  1 & 0 & 0 & 1 & 0 & 0 & 0 & z_{11} & z_{12} \\
  0 & 1 & 0 & 0 & 1 & 0 & 0 & z_{21} & z_{22} \\
  0 & 0 & 1 & 0 & 0 & 0 & 0 & z_{31} & z_{32} \\
  0 & 0 & 0 & 0 & 0 & 1 & 0 & z_{41} & z_{42} \\
  0 & 0 & 0 & 0 & 0 & 0 & 1 & 0 & 0 \\
\end{array}%
\right].$$

The non vanishing elements of the tensor are: $$\T_{123},\T_{213}$$
and  $R(\T)$ drops to $2$;

\item[c)]$g=1, g_{12}=g_{23}=1, l_{12}=l_{23}=0$, $g_{13}=2,
l_{13}=1,$ $\alpha_{12}=\alpha_{23}=0, \beta_{12}=\beta_{23}=2,
\alpha_{13}=1,\beta_{13}=1$. In this case $\Psi^4_{2,2,2}$ reduces to
 $$\Psi^4_{2,2,2}:=\left[%
\begin{array}{ccc|ccc|ccc}
  1 & 0 & 0 & 1 & 0 & 0 & 1 & 0 & 0 \\
  0 & 1 & 0 & 0 & 0 & 0 & 0 & 1 & 0 \\
  0 & 0 & 1 & 0 & 0 & 0 & 0 & 0 & a \\
  0 & 0 & 0 & 0 & 1 & 0 & 0 & 0 & b \\
  0 & 0 & 0 & 0 & 0 & 1 & 0 & 0 & c \\
\end{array}%
\right].$$ The non vanishing elements of the tensor are: $$
\T_{113},\T_{121},\T_{131},\T_{212},\T_{311},$$ $R(\T) = 4,$ and
again $\T$ is a linear combination of

$$\mathbf{e^1_1}\otimes
\mathbf{e^2_1}\otimes \mathbf{e^3_3},\quad \mathbf{e^1_1}\otimes
(\mathbf{e^2_2}+\mathbf{e^2_3}) \otimes \mathbf{e^3_1},\quad
\mathbf{e^1_2}\otimes \mathbf{e^2_1}\otimes \mathbf{e^3_2},\quad
\mathbf{e^1_3}\otimes \mathbf{e^2_1}\otimes \mathbf{e^3_1}.$$

\end{itemize}
In case $a)$ this shows that the border rank of the tensor is strictly less than its rank, i.e. $\underline{R}(\T) < R(\T).$
\end{example}
%\bibliographystyle{plain}
%\bibliography{20190630_bibdatabase}

\end{document}